\documentclass{conm-p-l}

\usepackage{amssymb}

\usepackage{graphicx}

\newtheorem{theorem}{Theorem}[section]

\newtheorem{proposition}[theorem]{Proposition}

\theoremstyle{definition}

\newtheorem{remark}[theorem]{Remark}

\numberwithin{equation}{section}

\newcommand\C{\mathbf{C}} 
\newcommand\R{\mathbf{R}}
\newcommand\Q{\mathbf{Q}}
\newcommand\Z{\mathbf{Z}}

\newcommand\stwomatr[4]{\left(\begin{smallmatrix}#1&#2\\  
                               #3&#4 \end{smallmatrix}\right)}

\newcommand\abs[1]{{\left|#1\right|}}

\newcommand\cuspforms{\mathcal{S}}

\newcommand\GammaN{{\Gamma(N)}} 

\newcommand\AAA{p}
\newcommand\BBB{q}
\newcommand\CCC{r}
\newcommand\ppp{\alpha}
\newcommand\qqq{\beta}
\newcommand\rrr{\gamma}

\begin{document}

\title
{Fake proofs for identities involving products of Eisenstein series}


\author{Kamal Khuri-Makdisi}
\address{Mathematics Department, 
American University of Beirut, Bliss Street, Beirut, Lebanon}
\email{kmakdisi@aub.edu.lb}
\thanks{August 27, 2018}




\begin{abstract}
In the workshop of the July 2016 Building Bridges 3 conference in
Sarajevo, I presented the results from a joint article with W. Raji
(Mathematische Annalen, 2017).  That article gave a proof of
various linear relations between products of two Eisenstein series on
$\GammaN$, including an interesting identity related to the action of
a Hecke operator on such a product.  The real proofs involve some care
to deal with issues of convergence.  In this note we give ``fake''
proofs for these identities, ignoring the convergence issues; some of
these fake proofs appeared in the workshop lecture as an amusing side
note before I sketched the real proofs.  Something in these fake
proofs is quite suggestive, even though the proofs themselves are
clearly invalid (and even produce wrong results).  It would be
interesting to understand what exactly is going on here.
\end{abstract}

\maketitle


\section{Introduction}
\label{section1}

The basic object of study in this note is the Eisenstein series of
weight $\ell \geq 1$ on $\GammaN$, with parameter $\lambda \in
N^{-1}\Z^2 / \Z^2 \subset \Q^2 / \Z^2$:
\begin{equation}
\label{equation1.1}
E_{\ell,\lambda}(z) = 
  \sum_{\substack{(a,b) \equiv \lambda \> (\text{mod}\> \Z^2)\\
                   (a,b) \neq (0,0)\\
                  }}
   (az+b)^{-\ell}.
\end{equation}
The above is not quite right if $\ell \in \{1,2\}$, as the above
series does not converge.  In that case, one evaluates the sum,
following Hecke, by replacing $(az+b)^{-\ell}$ by
$(az+b)^{-\ell} \abs{az+b}^{-s}$ for a complex parameter $s$, and then
setting $s=0$ after one has analytically continued the resulting sum in
$s$.  In the convergent case, when $\ell \geq 3$,
the Eisenstein series is a holomorphic function of $z$,  and this
holomorphy fortuitously still holds for $\ell = 1$; however, for
$\ell=2$, Hecke's summation procedure yields that $E_{2,\lambda}(z)$
is the sum of a nonholomorphic expression $-\pi/\text{Im } z$ (which
is the same for all $\lambda$) with a holomorphic function (which of
course depends on $\lambda$).

The product $E_{\ell,\lambda}E_{m,\mu}$ of two of these Eisenstein
series is then a form of weight~$\ell+m$ on $\GammaN$.  The
articles~\cite{BG2,BG3} prove a number of linear relations between such
products --- more precisely, they show that certain linear combinations 
of such products belong to the space $Eis_{\ell+m}$ of Eisenstein series of
weight~$\ell+m$.  These linear relations modulo~$Eis_{\ell+m}$ have
a structure that is reminiscent of the Manin relations between
periods of cusp forms; this was further codified in~\cite{Pasol}.
For example, in weight~$2$, we have the following relations:
\begin{equation}
\label{equation1.2}
\begin{split}
E_{1,\lambda} E_{1,\mu} + E_{1,\mu}E_{1,-\lambda} &= 0,\\
E_{1,\lambda} E_{1,\mu} + E_{1,\mu}E_{1,-\lambda-\mu} +
E_{1,-\lambda-\mu}E_{1,\lambda} &\equiv 0  \pmod{Eis_2}.\\
\end{split}
\end{equation}
The first identity, which is very simple, is analogous to the two-term
Manin relation involving $\stwomatr{0}{-1}{1}{0}$, while the second,
more interesting, identity, is
analogous to the three-term Manin relation involving
$\stwomatr{0}{1}{-1}{-1}$.  The explanation for this parallelism can
be found in~\cite{KKMWR}, where we show that the Petersson
inner product of any cusp form $f \in \cuspforms_2$ with
$E_{1,\lambda} E_{1,\mu} + E_{1,\mu} E_{1,-\lambda-\mu}
        + E_{1,-\lambda-\mu} E_{1,\lambda}$ 
expands to a combination of periods of $f$ and its transforms $f|_2 M$
for certain $M \in GL_2^+(\Q)$, and this combination of periods
vanishes precisely by the Manin relations.  The proof in that article,
which works for a similar identity in arbitrary weight, involves
carrying throughout the parameter $s$ in the Eisenstein series
$E(z,s)$, and controlling the analysis fairly carefully.  Our goal in
this note is to provide fake proofs of that and other results
from~\cite{KKMWR}, relying on intriguing identities between rational
functions, but with no attention paid to convergence.  It would be
interesting to find the connection between these intriguing identities
and the structure of the Manin relations, and to see what parts of the
fake proofs can be salvaged.
\begin{proof}[Fake proof of the second identity in~\eqref{equation1.2}]
Let us write $\nu = -\lambda-\mu$; hence we can assume given a
triple $(\lambda,\mu,\nu)$ for which $\lambda + \mu + \nu = (0,0)$.
Similarly, consider the set $T = T_{(\lambda,\mu,\nu)}$ of all 
triples $((a,b),(c,d),(e,f))\in (\Q^2)^3$ with
\begin{equation}
\label{equation1.3}
\begin{split}
(a,b) &\equiv \lambda \pmod{\Z^2},\\
(c,d) &\equiv \mu \pmod{\Z^2},\\
(e,f) &\equiv \nu \pmod{\Z^2},\\
(a,b) + (c,d) + (e,f) &= (0,0),\\
\text{None of } (a,b), (c,d), \text{ or } (e,f) &\text{ equals } (0,0).\\
\end{split}
\end{equation}
(The set $T$ is nonempty precisely because $\lambda + \mu + \nu =
(0,0)$, and the last condition of nonvanishing only matters if one of
$\lambda, \mu, \nu$ is zero in $\Q^2/\Z^2$.)  Then, ignoring all
issues of convergence, we formally have
\begin{equation}
\label{equation1.4}
E_{1,\lambda} E_{1,\mu}
   \equiv \sum_{((a,b),(c,d),(e,f)) \in T} 
         \frac{1}{az+b}\cdot \frac{1}{cz+d}
\pmod{Eis_2}.
\end{equation}
The reason is that once one chooses $(a,b)$ and $(c,d)$ arbitrarily in
the congruence classes of $\lambda$ and $\mu$, respectively, the pair
$(e,f) = - (a,b) - (c,d)$ is uniquely determined.  In the event that
$(e,f) = (0,0)$, which anyhow only occurs when $\lambda = -\mu$, the pairs
that we omit are those with $(c,d) = (-a,-b)$, which corresponds to
being off by $-E_{2,\lambda}$, which (ignoring its nonholomorphy) ``is''
an element of $Eis_2$.

We obtain similar expressions (modulo $Eis_2$) for
$E_{1,\mu}E_{1,\nu}$ and for $E_{1,\nu}E_{1,\lambda}$.  Thus, working
modulo $Eis_2$, we obtain
\begin{equation}
\label{equation1.5}
\begin{split}
&\qquad
E_{1,\lambda} E_{1,\mu} + E_{1,\mu}E_{1,\nu} + E_{1,\nu}E_{1,\lambda}\\
&\equiv
\sum_{((a,b),(c,d),(e,f)) \in T}
   \left[
     \frac{1}{(az+b)(cz+d)}
   + \frac{1}{(cz+d)(ez+f)}
   + \frac{1}{(ez+f)(az+b)}
   \right],\\
\end{split}
\end{equation}
and this last sum vanishes, thanks to the identity
\begin{equation}
\label{equation1.6}
\AAA+\BBB+\CCC = 0 \quad \implies \quad
 \frac{1}{\AAA\BBB} + \frac{1}{\BBB\CCC} + \frac{1}{\CCC\AAA} = 0,
\end{equation}
with $\AAA = az+b$, $\BBB=cz+d$, and $\CCC=ez+f$.
\end{proof}

It is not apparent to me how to salvage the above fake proof, for
instance by summing over elements of $T$ in a particular order so as
to obtain convergence, in the style of
Eisenstein~\cite{WeilEllipticFns}.  Indeed, 
when $\lambda,\mu,\nu$ are all nonzero, then our fake proof would
imply that $E_{1,\lambda} E_{1,\mu} + E_{1,\mu}E_{1,\nu} +
E_{1,\nu}E_{1,\lambda}$ is actually zero, which is not the case; its
expression as an explicit weight~2 Eisenstein series is known.

\section{The identity in higher weight}
\label{section2}

We now turn to the case of general weight~$k \geq 2$.  The analog of
the two-term Manin relation is again simple, and can be found in
equation~(2.23) of~\cite{KKMWR}.  The interesting three-term relation
in higher weight amounts to the following.

\begin{proposition}[Theorem~2.8 of~\cite{KKMWR}]
\label{proposition2.1}
Let $\lambda,\mu,\nu$ satisfy $\lambda+\mu+\nu=0$, as before.  Let
$\ppp,\qqq,\rrr\in\C$ satisfy $\ppp+\qqq+\rrr=0$ (these can be thought
of as formal variables).  Let $k \geq 2$.  Then the following expression is
orthogonal to all cusp forms $f \in \cuspforms_k$:
\begin{equation}
\label{equation2.1}
\sum_{\substack{\ell+m=k\\ \ell,m \geq 1}}
\ppp^{\ell-1} \qqq^{m-1}
E_{\ell,\lambda} E_{m,\mu}
 +
 \sum_{\substack{\ell+m=k\\ \ell,m \geq 1}}
\qqq^{\ell-1} \rrr^{m-1}
E_{\ell,\mu} E_{m,\nu}
 +
 \sum_{\substack{\ell+m=k\\ \ell,m \geq 1}}
\rrr^{\ell-1} \ppp^{m-1}
E_{\ell,\nu} E_{m,\lambda}.
\end{equation}
Morally speaking, this means that the expression~\eqref{equation2.1}
should belong to $Eis_k$, but the presence of nonholomorphic $E_2$
terms complicates the statement somewhat.
\end{proposition}
\begin{proof}[Fake proof]
We again ignore all issues of convergence, and work modulo ``$Eis_k$''
(ignoring the nonholomorphy coming from any $E_2$). 
Analogously to~\eqref{equation1.4}, we formally write the first term 
$\sum_{\ell+m=k} \ppp^{\ell-1} \qqq^{m-1} E_{\ell,\lambda}(z) E_{m,\mu}(z)$
as the following sum over $((a,b),(c,d),(e,f)) \in T$ and $\ell,m \geq 1$
with $\ell+m=k$:
\begin{equation}
\label{equation2.2}
\sum_{((a,b),(c,d),(e,f))} \sum_{\ell,m}
       \frac{\ppp^{\ell-1} \qqq^{m-1}}{(az+b)^{\ell} (cz+d)^m}
=
\sum_{((a,b),(c,d),(e,f))} \sum_{\ell,m}
       \frac{\ppp^{\ell-1} \qqq^{m-1}}{\AAA^{\ell} \BBB^m},
\end{equation}
using the notation of $\AAA,\BBB,\CCC$ as in~\eqref{equation1.6} and the
sentence that follows it.  Some manipulation with the finite
geometric series over $\ell,m$ then gives us the following congruence
modulo ``$Eis_k$'':
\begin{equation}
\label{equation2.3}
\sum_{\substack{\ell+m=k\\ \ell,m \geq 1}}
        \ppp^{\ell-1} \qqq^{m-1} E_{\ell,\lambda} E_{m,\mu}
\equiv
\sum_{((a,b),(c,d),(e,f)) \in T} 
 \frac
    {\left(\frac{\ppp}{\AAA}\right)^{k-1} - \left(\frac{\qqq}{\BBB}\right)^{k-1}}
    {\ppp\BBB - \qqq\AAA}.
\end{equation}
Taking simultaneous cyclic permutations of $(\ppp,\qqq,\rrr)$ and $(\AAA,\BBB,\CCC)$, we
obtain similar (purely formal) expressions for the second and third terms
of~\eqref{equation2.1}.  But now a miracle occurs: the identities
$\ppp+\qqq+\rrr=0$ and $\AAA+\BBB+\CCC=0$ imply that
\begin{equation}
\label{equation2.4}
\ppp\BBB-\qqq\AAA = \qqq\CCC-\rrr\BBB = \rrr\AAA-\ppp\CCC.  
\end{equation}
(An amusing way to avoid verifying the above fact algebraically is to
stick to real variables, and use the cross product in $\R^3$: in that
case, the vectors $(\ppp,\qqq,\rrr)$ and $(\AAA,\BBB,\CCC)$ both lie in the plane
orthogonal to $(1,1,1)$, so their cross product is parallel to
$(1,1,1)$, which gives precisely the equalities in~\eqref{equation2.4}.)
Now adding up the cyclic permutations of the expressions on the right
hand side of~\eqref{equation2.3} gives a sum of the cyclic
permutations of 
$\left(\frac{\ppp}{\AAA}\right)^{k-1} - \left(\frac{\qqq}{\BBB}\right)^{k-1}$ over
the \textbf{same} common denominator, and this immediately gives the
desired sum of zero!
\end{proof}

The identities between rational functions that we have used in the
above two fake proofs, namely $1/\AAA\BBB + 1/\BBB\CCC + 1/\CCC\AAA = 0$ and the
analogous cyclic sum involving also $\ppp,\qqq,\rrr$ in higher weight, date
back to Eisenstein; a good reference for this is Chapter~II,
section~2 and Chapter~IV, section~1 of~\cite{WeilEllipticFns} (but
note that Weil uses $r=p+q$, whereas we use $\CCC=-\AAA-\BBB$).  The
identities there, which are proved by a partial fraction decomposition
and/or successive differentiation, may look more complicated than
ours, particularly since they involve various sums and binomial
coefficients.  The relation to the identities we used above with auxiliary
variables $\ppp,\qqq,\rrr$ are however straightforward: write $\rrr =
-\ppp-\qqq$ in our identities, expand everything into a polynomial in
$\ppp$ and $\qqq$, then equate the coefficients of the same monomial
$\ppp^\ell \qqq^m$ on both sides to obtain the identities
in~\cite{WeilEllipticFns}.  However, those identities, just like ours,
always end up involving some terms with only a first power or square
of $\AAA$ (or $\BBB$, or $\CCC$) in the denominator.  This appears to
make the convergence difficult to control, even if one sums both sides in
Eisenstein style, with a sum $\sum_{(m,n) \in \Z^2}$ being carried out
as $\lim_{M \to \infty} \sum_{m=-M}^M \lim_{N \to \infty}
\sum_{n=-N}^N$.  For that reason, the treatment in Chapter~IV
of~\cite{WeilEllipticFns}, following Eisenstein, proceeds by a
somewhat different route.

\section{An example related to Hecke operators}
\label{section3}

Another result of interest relates to the trace from
$\Gamma(NM)$ to a lower level $\Gamma(M)$ of certain products of
Eisenstein series.  The basic expression for which we derive an
identity 
in~\cite{KKMWR}
is a sum of the form
$\sum_{\tau \in N^{-1}\Z^2/\Z^2}
 E_{\ell,\lambda+\tau} E_{m, \mu-S\tau}$, as explained in Section~4
 of that article.  
Here $\lambda,\mu \in M^{-1}\Z^2/\Z^2$, and $S \in \Z$.  The variable 
$\tau \in N^{-1}\Z^2/\Z^2$ in the sum can be thought of as a sum over
all $N$-torsion points of the elliptic curve given analytically as
$\C/(\Z z + \Z)$.  In this note, we will only deal with one example,
but the result holds generally, as does the fake proof, in whatever
sense fake proofs can be said to hold.  The combinatorics are
again reminiscent of the combinatorics one obtains when one computes
Hecke operators on spaces of modular symbols, and involve sublattices
of $\Z^2$ and the convex hull of the lattice points in the first
quadrant; the references for this are Theorem~3.16 of~\cite{BG2} and
its proof, Lemma~7.3 of~\cite{BG3}, and Subsection~2.3 and Section~3
of~\cite{Merel}.  Here we will just illustrate these phenomena for the
case $N=5$ and $S=3$, and make the connection with fake proofs based
on interesting identities of rational functions.

We thus fix $\lambda,\mu \in \Q^2$ (where usually only their image in
$\Q^2/\Z^2$ matters).  The exact level of $\lambda,\mu$, i.e., their
denominator, $M$ is immaterial, since the formulas we obtain are
insensitive to $M$.  We then consider just the following identity,
which is a special case of Proposition~4.1 of~\cite{KKMWR} (and the
notation $L_{\lambda,\mu,\ppp,\qqq}$ is taken from there as well).

\begin{proposition}
\label{proposition3.1}
Write
$L_{\lambda,\mu,\ppp,\qqq} = \sum_{\substack{\ell+m=k\\ \ell,m \geq 1}}
        \ppp^{\ell-1} \qqq^{m-1} E_{\ell,\lambda} E_{m,\mu}$
for the expression in weight~$k$ that has appeared repeatedly in
Section~\ref{section2}.  Then, modulo ``$Eis_k$'' as usual, we have
\begin{equation}
\label{equation3.1}
\begin{split}
(1/5) & \sum_{\tau \in 5^{-1}\Z^2/\Z^2}
                L_{\lambda + \tau, \mu - 3\tau, \ppp, \qqq}\\
&\equiv
L_{5\lambda,3\lambda + \mu,5\ppp, 3\ppp+\qqq}
+ L_{3\lambda+\mu, \lambda+2\mu, 3\ppp+\qqq, \ppp+2\qqq}
+ L_{\lambda+2\mu, 5\mu, \ppp+2\qqq, 5\qqq}.\\
\end{split}
\end{equation}
\end{proposition}
\begin{remark}
\label{remark3.2}
The vectors $(5,0)$, $(3,1)$, $(1,2)$, and $(0,5)$ in sequence are
obtained by taking the convex hull of the nonzero points in the first
quadrant of the sublattice 
$\{(x,y) \mid x-3y \equiv 0 \pmod{5}\}$ of $\Z^2$,
as in the figure below.  Any pair of consecutive vectors has a
determinant of $5$, the index.  This is described further in the
references mentioned above.
\end{remark}
\includegraphics[scale=0.9, trim=0mm 13mm 0mm 20mm, clip]{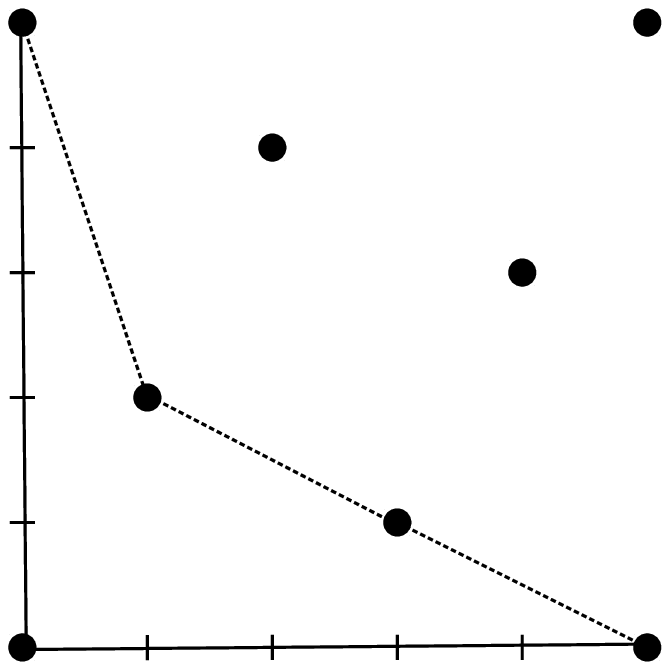}
\begin{proof}[Fake proof of Proposition~\ref{proposition3.1}]
It is clearer if we first restrict to $k=2$, in which case
everything related to $\ppp,\qqq$ can be omitted, and
$L_{\lambda,\mu,\text{anything}} = E_{1,\lambda}E_{1,\mu}$.  We
formally expand $\sum_\tau E_{1,\lambda+\tau}E_{1,\mu-3\tau}$
using as usual $v = (a,b) \equiv \lambda+\tau$ and 
$w = (c,d) \equiv \mu-3\tau \pmod{\Z^2}$, taking into account also the
sum over $\tau$.  Write as usual $\AAA = az+b$ and $\BBB = cz+d$;
these depend linearly on $v$ and $w$.  Our left hand side is thus the
sum of all terms $1/(5\AAA\BBB)$, as $(v,w) = (\lambda+v',\mu+w') \in
\Q^4$ ranges over all possible shifts of $(\lambda,\mu)$ by the
lattice $\Lambda = \{(v',w') \in 5^{-1}\Z^4 \mid 3v'+w' \in \Z^4\}$. 

For each two consecutive vectors in the list $\{(5,0), (3,1), (1,2),
(0,5)\}$, say for example the vectors $(3,1)$ and $(1,2)$, one can
see that as $(v',w')$ varies over $\Lambda$, the resulting combination
$(3v'+w', v'+2w')$ (made using the coefficients of the two consecutive
vectors) varies precisely over all of $\Z^4$.  It follows that the
pair of values $(3\AAA+\BBB, \AAA+2\BBB)$ varies over the
terms in such a way that the (nonconvergent, as always) sum of all the
products $(3\AAA+\BBB)^{-1}(\AAA+2\BBB)^{-1}$ yields
$E_{1,3\lambda+\mu}E_{1,\lambda+2\mu}$.  It now remains to make use of the
hopefully impressive identity
\begin{equation}
\label{equation3.2}
\frac{1}{5\AAA\BBB}
   = \frac{1}{(5\AAA)(3\AAA+\BBB)}
     + \frac{1}{(3\AAA+\BBB)(\AAA+2\BBB)}
     + \frac{1}{(\AAA+2\BBB)(5\BBB)}
\end{equation}
to conclude the fake proof for $k=2$.  In all this, we have blithely
ignored the fact that in all our sums, we omitted any terms that 
look like $1/0$, which may have introduced correction terms that with
luck will belong to $Eis_2$; as mentioned at the end of
Section~\ref{section1}, however, the issues with convergence seem to
produce further unavoidable corrections from $Eis_2$, even if the
above formal argument has not omitted any terms (e.g., if all of
$5\lambda,3\lambda+\mu,\lambda+2\mu,5\mu$ are nonzero in $\Q^2/\Z^2$).

We note that the identity~\eqref{equation3.2} may become more apparent if we
observe that
\begin{equation}
\label{equation3.3}
\begin{split}
 \frac{1}{(5\AAA)(3\AAA+\BBB)}
&= \frac{1/\BBB}{5\AAA} - \frac{3/(5\BBB)}{3\AAA+\BBB},\\
\frac{1}{(3\AAA+\BBB)(\AAA+2\BBB)}
&= \frac{3/(5\BBB)}{3\AAA+\BBB} - \frac{1/(5\BBB)}{\AAA+2\BBB},\\
\frac{1}{(\AAA+2\BBB)(5\BBB)}
&= \frac{1/(5\BBB)}{\AAA+2\BBB} - \frac{0}{5\BBB},\\
\text{generally, }
\frac{1}{(a\AAA+b\BBB)(c\AAA+d\BBB)}
&= \frac{a/((ad-bc)\BBB)}{(a\AAA+b\BBB)}
  - \frac{c/((ad-bc)\BBB)}{(c\AAA+d\BBB)}.
\end{split}
\end{equation}
Thus when we add the terms coming from each consecutive pair vectors
in the list $\{(N,0), \dots, (0,N)\}$, all the middle terms cancel,
and we are left with $1/(N\AAA\BBB)$.

We now turn to the fake proof for arbitrary weight~$k$.  For general $k$,
the sum giving the left hand side of~\eqref{equation3.1},
$(1/5)\sum_\tau L_{\lambda+\tau,\mu-3\tau,\ppp,\qqq}$, is a
sum not merely of $1/(5\AAA\BBB)$, but rather, as
in~\eqref{equation2.3}, of  
$[(\ppp/\AAA)^{k-1} - (\qqq/\BBB)^{k-1}]/[5(\ppp\BBB - \qqq\AAA)]$.
The sum runs over the same collection of $\AAA,\BBB$ as before, and the
same type of combinations (using consecutive vectors in the list
$\{(5,0), \dots\}$) relate the lattice $\Lambda$ to $\Z^4$,
specifically for the terms appearing on the right hand side
of~\eqref{equation3.1}.  One obtains that the three $L$ 
expressions there amount to summing each of the following three terms:
\begin{equation}
\label{equation3.4}
\begin{split}
&\frac{(5\ppp/5\AAA)^{k-1} - \bigl((3\ppp+\qqq)/(3\AAA+\BBB)\bigr)^{k-1}}
     {(5\ppp)(3\AAA+\BBB) - (3\ppp+\qqq)(5\AAA)},
\\
&\frac{\bigl((3\ppp+\qqq)/(3\AAA+\BBB)\bigr)^{k-1}
        - \bigl((\ppp+2\qqq)/(\AAA+2\BBB)\bigr)^{k-1}}
     {(3\ppp+\qqq)(\AAA+2\BBB) - (\ppp+2\qqq)(3\AAA+\BBB)},
\\
&\frac{\bigl((\ppp+2\qqq)/(\AAA+2\BBB)\bigr)^{k-1} - (5\qqq/5\BBB)^{k-1}}
     {(\ppp+2\qqq)(5\BBB) - (5\qqq))((\AAA+2\BBB))}.
\\
\end{split}
\end{equation}
Once again, a minor miracle occurs in that the denominators of the
terms are all equal to the same expression, namely
$5(\ppp\BBB - \qqq\AAA)$, so we formally get the desired left hand
side.  The same phenomenon happens in general, not just for $N=5$.
\end{proof}


\bibliographystyle{amsalpha}
\bibliography{sarajevo2}

\end{document}